\documentclass{imsart}
\usepackage[margin=1.75in]{geometry}

\usepackage{fancyhdr}
\usepackage{stackengine}
\usepackage{microtype}
\usepackage[textsize=footnotesize]{todonotes}
\usepackage{xcolor}
\usepackage{hyperref}

\usepackage{mathrsfs}
\usepackage{graphicx}
\usepackage{amsmath}        
\usepackage{amssymb}
\usepackage{amsthm}
\usepackage{bm}
\usepackage[mathscr]{euscript}
   
        \usepackage[countmax]{subfloat}  
\usepackage{fourier}
\usepackage{xcolor}
\usepackage{mathtools}
\usepackage{enumerate}
\usepackage{enumitem}
\counterwithin{figure}{section}
\usepackage[numbers]{natbib}
\usepackage{url}
\usepackage{verbatim}
\usepackage{booktabs}
\usepackage{placeins}
\usepackage{multirow}
\usepackage{subcaption}
\usepackage[cmtip,all]{xy}
\usepackage{aliascnt}
\usepackage{ifthen}
\hypersetup
{
    pdfauthor={Anne van Delft},
    pdfkeywords={Geometric and measure-theoretic characterizations of the measure-determining property of balls in metric spaces}
}
\newcounter{alphcount}
{\begin{list}{{\upshape(}\alph{alphcount}\/{\upshape)\ }}%
             {\usecounter{alphcount}\labelwidth1.5em%
              \bigmargin2em\labelsep0.5em\topsep0.25em plus 0.5ex%
              \itemsep0.25em plus 0.5ex\parsep0em}}{\end{list}}
{\begin{list}{{\upshape(#1\arabic{alphcount})\hfill}}%
             {\usecounter{alphcount}\labelwidth2.5em%
              \bigmargin2.5em\labelsep0em\topsep0.25em plus 0.5ex%
              \itemsep0.25em plus 0.5ex\parsep0em}}{\end{list}}
\newcommand\tageq{\addtocounter{equation}{1}\tag{\theequation}}

\newcommand{\mynewtheorem}[2]{
  \newaliascnt{#1}{dummy}
  \newtheorem{#1}[#1]{#2}
  \aliascntresetthe{#1}
  \expandafter\def\csname #1autorefname\endcsname{#2}
}
 
\theoremstyle{plain}
  \mynewtheorem{thm}{Theorem}
  \mynewtheorem{Proposition}{Proposition}
  \mynewtheorem{Corollary}{Corollary}
  \mynewtheorem{assumption}{Assumption}
  \mynewtheorem{lemma}{Lemma}
\theoremstyle{definition}
  \mynewtheorem{definition}{Definition}
  \mynewtheorem{example}{Example}
\theoremstyle{definition}
  \mynewtheorem{Remark}{Remark}

\newcommand{\mb}[1]{\mathbb{#1}}
\newcommand{\mc}[1]{\mathcal{#1}}

\newcommand{\rnum}{\mathbb{R}}

\newcommand{\flo}[1]{\lfloor#1\rfloor}
\newcommand{\F}{\mathcal{F}}

\allowdisplaybreaks

\newcommand{\longsquiggly}{\xymatrix@C=1.5em{{}\ar@{~>}[r]&{}}}
\newcommand{\medsquiggly}{\xymatrix@C=1.2em{{}\ar@{~>}[r]&{}}}

\newcounter{relctr} 
\everydisplay\expandafter{\the\everydisplay\setcounter{relctr}{0}} 

\AtBeginDocument{}

\newboolean{withproof}
\setboolean{withproof}{true}

\begin{document}

\begin{frontmatter}
    
\title{
Geometric and measure-theoretic characterizations of the measure-determining property of balls in metric spaces
}
\maketitle

\begin{aug}
\author{\fnms{Anne} \snm{van Delft}\ead[label=e1,mark]{anne.vandelft@columbia.edu}}

\address[A]{Department of Statistics, Columbia University, 1255 Amsterdam Avenue, New York, NY 10027, USA.\\ \printead{e1}}
\end{aug}

\begin{abstract}
The question of what structural properties drive the collection of open or closed balls to uniquely determine Borel probability measures on a metric space---the \textit{ball} property---has a long history in measure theory. While various sufficient conditions have been known since the 1970s, and it was shown by Preiss and Ti\v{s}er (1991) that it holds for all separable Banach spaces, Davies (1971) constructed a famous counterexample that shows it can fail on compact metric spaces. It was only recently established in van Delft (2026) that metric spaces admitting an isometric Hilbertian embedding must possess the ball property. Building on the functional-analytic and geometric foundations in earlier work, this article provides a comprehensive structural and measure-theoretic characterization of the ball property. First, we establish a necessary and sufficient condition for the ball property to fail: the collection of balls must possess sufficient combinatorial capacity, which we define as \textit{charged dyadic shattering}. Second, by characterizing the ball property via $\lambda$-systems, we place classical vector space results within a unified context. Building on this to understand spaces that have the topological capacity but lack linear or geometric rigidity required to force ball-annihilating measures to zero, we relate the metric outer measure to the cylinder outer measure on countable products of Polish spaces, deriving necessary and sufficient measure-theoretic conditions for the ball property to hold. This allows us to derive a structural taxonomy on the decoupling of the metric outer measure from the cylinder outer measure. This is subsequently used to deconstruct Davies' counterexample. Finally, we introduce a minimal   generalization of the ball property---the \textit{$n$th order ball property}---and show that such classes  can fully recover measure determination in highly pathological constructions like Davies' counterexample. 
\end{abstract}



\end{frontmatter}

\section{Introduction}\label{sec1}

Uniquely identifying a measure via a subcollection of sets is a foundational problem in measure theory dating back to the work of Carath{\'e}odory \cite{Car1914}, Sierpi{\'n}ski \cite{Sie1928} and Dynkin \cite{Dynkin}. Historically, literature relied heavily on vector space structure and duality to find such suitable collections \citep[see e.g.,][]{Levy1925,CramerWold1936,kolmogorov1933,Mourier,Prokhorov1956}.

With the highly influential works of Parthasarathy \cite{Part67} and Billingsley \cite{bil68} in the late 1960s, the probability community became interested in finding suitable sets on general metric spaces to uniquely determine measures. The most natural candidate for a purely geometric separating class of sets is the collection of open (or closed) balls. Unlike in $\rnum$, the collection of balls fails to form a $\pi$-system.  This led to two deeply related questions raised by Tops{\o}e and Goffman  \citep[see][]{davies1971}. Tops{\o}e asked whether balls uniquely determine finite Borel measures; that is, if $\mu(\bar{B})=\nu(\bar{B})$ for all closed balls $\bar{B} \subseteq S$, does this imply that $\mu= \nu$ on the Borel $\sigma$-algebra $\mc{B}(S)$? We will henceforth refer to this property as \textit{the ball property}. Concurrently, Goffman posed the question of whether measures can be approximated up to arbitrary precision by finite disjoint collections of small balls.
Because strong covering properties \cite{Besicovitch45} imply measure determination, these problems are closely linked; indeed, measure determination holds for separable metric spaces that are directionally limited in the sense of Federer \cite{federer1969} \citep[see e.g.,][]{buet2016}. Anderson (unpublished manuscript; see e.g.,\cite{hoffmann1975}) proved that spaces with finite doubling dimension possess the ball property. Christensen \cite{christensen1980} showed that if there exists a uniform measure on the metric space, the collection of open balls forms a separating class. After this, positive results were obtained  for large classes of Banach spaces \citep[see e.g.,][]{hoffmann1975, riss1989,preisstiser1982}, with Preiss and Ti\v{s}er  ultimately showing in 1991 \cite{PreissTisier91} that all separable Banach spaces possess the ball property.  
In parallel, Davies \cite{davies1971} had already demonstrated in 1971 that both properties can fail in compact metric spaces. 

It remained therefore unclear what intrinsic structural properties a metric space should possess for the collection of open (or, equivalently, closed)  balls to determine measures. Studen{\'y} \cite{studeny1983} proved that Davies' counterexample could be isometrically embedded into a compact Abelian group with a translation-invariant metric. Although Christensen \cite{christensen1980} showed that such groups possess the ball property, Studen{\'y}'s result revealed that isometric subsets do not inherit this property. In the same paper, Studen{\'y} therefore asked if Davies' space could be isometrically embedded into a Hilbert space.  Only recently, van Delft \cite{vD26} provided a definitive negative answer to Studen{\'y}'s question by proving that spaces isometrically embeddable into a Hilbert space form a strict subclass of the collection of separable metric spaces that possess the ball property. This confirms that the Hilbert space geometry is effectively too rigid to `hide' Davies-type counterexamples, automatically forcing measure-determination on all its isometric subsets.  

While this does settle the embedding question, the main question is still unresolved: what geometric and/or measure-theoretic properties drive the ball property in general metric spaces? This article provides a more comprehensive understanding of the property. In \autoref{sec2}, we establish a necessary and sufficient condition for the failure of the ball property based on the localized combinatorial capacity of the collection of open balls $\textbf{B}$. We prove that the ball property fails on a separable metric space if and only if the space exhibits \textit{charged dyadic shattering} (\autoref{thm:dyadicallyshatter}). This means that the collection of balls must have sufficient combinatorial freedom to allow  recursive spatial partitioning (dyadic shattering) without forcing every ball-annihilating signed measure (or charge) to identically vanish. While topological dyadic shattering is satisfied for many metric spaces, including non-trivial separable Banach spaces, the geometric rigidity of balls prevents the space from sustaining a non-zero charge across every branch. To investigate these cases, we introduce a further characterization of the ball property via its generated $\lambda$-system. We show this system contains the separating class that was used in \cite{PreissTisier91} to derive the ball property, thereby placing Preiss and Ti\v{s}er's fundamental theorem for Banach spaces within a unified, general measure-theoretic framework. 

Separable Banach spaces or spaces that are isometric to a subset of a Hilbert space are the only known infinite-dimensional cases where the geometric rigidity successfully prevents charged dyadic shattering. To understand what can be concluded for infinite-dimensional spaces under more complex geometries in a natural setting, we consider countable products of Polish spaces in \autoref{sec:sec3}. By relating the cylinder outer measure to the centered metric outer measure, a necessary and sufficient measure-theoretic condition for the ball property to hold on a countable product space is obtained in \autoref{thm:ballprod}. From this, it becomes clear that without Hilbert geometry or the rigidity of a vector space structure, the lack of geometric homogeneity combined with the combinatorial capacity of the collection of open (or closed) balls allows metric outer measures to effectively decouple from the cylinder outer measures. 
Consequently, we conjecture that in these cases, the failure of the ball property is inherently measure-dependent. More specifically, that there always exists a Borel probability measure for which the metric outer measure can be decoupled from the cylinder outer measure, precluding a purely geometric or topological characterization for the space as a whole.  These results are then used to deconstruct the pathology of Davies' classic counterexample, illustrating the exact mechanics behind it. 

In \autoref{sec:sec4}, we consider (minimal) structure-preserving generalizations of the ball property for spaces where the standard ball property fails. This is of interest, especially in probability and statistical theory for complex stochastic systems where simple structure-preserving separating classes are of the essence. We introduce the \textit{$n$th-order ball property}, that is, the property that the collection of intersections of at most $n$ open balls is a separating class of sets. In particular, we show that Davies' metric space satisfies the second-order ball property, demonstrating that minimal combinatorial generalization can fully recover measure determination, even in certain highly pathological constructions. 

Finally, these findings are summarized into a structural taxonomy in \autoref{sec:sec5}, showing how distinct structural conditions---metric geometry, combinatorial capacity and vector space structure---interact to guarantee or preclude measure-determination via the collection of balls, thereby exposing the mechanics driving the decoupling of outer measures ($\mu^\star_{\textbf B} > \mu^\star_{\mc{C}}$).

\section{Structural conditions for the ball property}\label{sec2}

\subsection{Local combinatorial capacity and dyadic shattering}\label{sec:sec21}

For a separable metric space $(S,d_S)$, denote by $B(x,r) = \{y \in S: d_S(x,y) < r\}$ the open ball centered at $x$ of radius $r$, and ${\bf B}=\{B(x,r): x \in S, r >0\}$ the collection of open balls. We start with the following functional-analytic necessary and sufficient condition given in \cite{vD26}:
\begin{lemma}\label{lemm:ballsspan}
    Let $(S, d_{S})$ be a separable metric space. Then the linear span of $\{1_B: B \in {\bf B}\}$ is dense in $L^1(S,\mc{B}(S),\mu)$ for every Borel probability measure $\mu$ if and only if 
    ${\bf B}$ forms a separating class. 
\end{lemma}

For the ball property to fail there must thus exist a Borel probability measure $\mu$ and a nonzero $g \in L^\infty(S,\mc{B}(S),\mu)$ such that the signed measure $\eta=g d\mu$, annihilates the balls, i.e., $\eta(B)=0$ for all $B \in {\bf B}$, while for some $f \in L^1(S,\mc{B}(S),\mu)$
    \[
    \eta(f)\neq 0. \tageq \label{eq:annif}
    \]
In particular, it is necessary that there exists a collection of balls, say $\{B_1, \ldots, B_n\}$, such that \eqref{eq:annif} holds for $f= 1_{\cap_{i=1}^n B_i}$. Indeed, if the collection of intersections lies in the closure of the linear span of the balls, then by the $\pi$-$\lambda$ theorem it is straightforward to show that the ball property must hold. 
In order to formalize a necessary and sufficient characterization of the ball property, we introduce a localized, tree-structured notion of shattering by  collections of open balls.

\begin{definition}[Dyadic shattering]\label{def:dyadicallyshat}
    Let $(S,d_S)$ be a metric space and let ${\textbf B}$ be the collection of open balls. A non-empty  set $A \subseteq S$ is said to be (topologically) \textit{dyadically shattered} by ${\textbf B}$ if there exists a family of non-empty subsets $(I_s)_{s \in 2^{<\mb{N}}}$ of $A$, and a corresponding family of open balls $(B_s)_{s \in 2^{<\mb{N}}} \subseteq {\textbf B}$, such that $I_\emptyset = A$, and such that for every finite binary string $s \in 2^{<\mb{N}}$:
\begin{enumerate}[label=(\roman*)]
    \item $I_{s^\frown 1} = I_s \cap B_s \neq \emptyset$,
    \item $I_{s^\frown 0} = I_s \cap B_s^\complement \neq \emptyset$,
\end{enumerate}
where $s \frown i$ denotes the concatenation of $s$ with $i \in \{0,1\}$. 
We will say that a metric space $(S,d_S)$ exhibits  \textit{charged  dyadic shattering} if there exists a non-empty Borel measurable set $A \subseteq S$ that is dyadically shattered by ${\textbf B}$ and  there exists a finite signed Borel measure $\eta$  annihilating ${\textbf B}$, such that  every generated subset $I_s$ satisfies $|\eta|(I_s)>0$.
\end{definition}

Note that topological dyadic shattering is a natural consequence in any metric space containing a non-empty set with no isolated points. Every relative open neighborhood in such a set contains infinitely many points, allowing one to recursively bisect any region with open balls. Since the support of a  diffuse measure on a Polish space must be a perfect set, this property is very common and holds, for example, for $\rnum^d$. However, the charged dyadic shattering property is much more delicate. Charged dyadic shattering tells us that the metric space does not just have the topological capacity to dyadically shatter a non-empty subset of $S$, but also that it has the geometric capacity to support a nonzero `charge' throughout this recursive partition  that completely cancels on the balls themselves. The choice of the term `charge' recognizes the classical nomenclature  of Kolmogorov and Fomin \cite{kolfom}, who refer to countably additive signed measures as `charges'.

\begin{thm}\label{thm:dyadicallyshatter}
Let $(S,d_S)$ be a separable metric space. The ball property fails on $(S,d_S)$ if and only if it exhibits charged dyadic shattering.
\end{thm}

The proof of this theorem relies on the following result, which is proven in the appendix.

\begin{lemma}\label{lem:atomforce}
    Let $(S,d_S)$ be a separable metric space and let $\nu$ be a finite nonnegative Borel measure with $\nu(A) = k >0$ for some nonempty Borel measurable set $A$. If $\nu(A \cap B) \in \{0,k\}$ for all $B \in \textbf{B}$ then $\nu$ restricted to $A$ is atomic, i.e., $\nu\vert_A =  k \delta_{x_0}$ for some unique $x_0 \in {A}$.
\end{lemma}

\begin{proof}[Proof of \autoref{thm:dyadicallyshatter}]
For the ball property to fail, there must exist some minimal order $n^\star\ge 2$ and a collection of balls $\{B_1, \ldots, B_{n^\star}\}$ such that
\[
\eta\Big(\bigcap_{i=1}^{n^\star} B_i\Big)=c \neq 0,
\]
and such that all intersections up to index $n^\star-1$ still have $\eta$ measure zero. For each $n$, the balls generate the collection of atoms $\mc{A}_n$ of a Boolean algebra that partitions the space into sets of the form:
    \[
    A = \Big(\bigcap_{j \in J} B_j\Big) \bigcap \Big(\bigcap_{i \in J^\complement} B^\complement_i\Big),  
    \]
    where $J \subseteq \{1,2, \ldots,n\}$. Note that $|\mc{A}_n|=2^n$. Evaluating $\eta$ on this collection yields the collection of real variables
$\{\eta(A): A \in \mc{A}_n\}$ where 
\[
\eta(A) = \sum_{I \subseteq J^\complement}(-1)^{|I|}\eta\Big(\bigcap_{\ell \in J \cup I}  B_\ell\Big).
\]
For any $n < n^\star$, we have $|J \cup I| \le n <n^\star$ and so $\eta(A) = 0$ for all $A \in \mc{A}_{n}$. If $n=n^\star$ then every proper subset $I \subset J^\complement$ produces an intersection of order $|J \cup I| < n^\star$, and therefore has $\eta$ measure zero. Furthermore, by assumption: $
\eta\Big(\bigcap_{\ell=1}^{n^\star}  B_\ell\Big)= c \neq 0$. The only term that  therefore survives for any $A \in \mc{A}_{n^\star}$ is the term where $I=J^\complement$, i.e., 
\[
\eta(A) =(-1)^{|J^\complement|} \eta\Big(\bigcap_{\ell=1}^{n^\star}  B_\ell\Big).
\]
Hence, we must have
\[
\eta(A) = \pm c \neq 0, \quad \forall A \in \mc{A}_{n^\star}.
\]
This means that none of the $2^{n^\star}$ atoms can be  empty, i.e., the collection of $n^\star$ balls must partition the space into  $2^{n^\star}$ non-empty pieces. This means the collection of $n^\star$ balls is dually  shattered \citep[see e.g.,][]{assouad1983,matousek2002}. Furthermore, we show that the measure $\eta$ must be diffuse. First countability of a metric space implies 
\[
\{x\} = \bigcap_{k=1}^\infty B(x,r_k)
\]
where $\{r_k\}_{k \in \mb{N}}$ is a decreasing sequence of real numbers such that $r_k \to 0$ as $k \to \infty$. Then, by continuity from above and by the fact that $\eta$ annihilates the collection of balls
\[
\eta(\{x\}) = \lim_{k \to \infty}\eta(B(x,r_k)) = 0. \tageq\label{eq:cfa_sing}
\]
Hence, the measure $\eta$ must be non-atomic. 
Fix an atom $I_\emptyset \in \mc{A}_{n^\star}$ so that $\eta(I_\emptyset) = \pm c \neq 0$. By the Hahn decomposition $\eta=\eta^+-\eta^-$, at least one of the two mutually singular parts must give strictly positive mass to $I_\emptyset$. Assume without loss of generality that $\eta^+(I_\emptyset) = k > 0$. Since $\eta$ is a finite diffuse measure, $\eta^+$ is a finite, nonnegative diffuse measure. Consider the restricted measure $\eta^+\vert_{I_\emptyset}(E) = \eta^+(I_\emptyset \cap E)$ for $E \in \mc{B}(S)$. Because it is non-atomic, we can apply the contrapositive of \autoref{lem:atomforce} to $\eta^+\vert_{I_\emptyset}$, which implies that there must exist a ball $B_\emptyset \in \textbf{B}$ such that \[ \eta^+(I_\emptyset \cap B_\emptyset) \notin \{0,k\}. \] This ball splits $I_\emptyset$ into two new components $I_1 = I_\emptyset \cap B_\emptyset$ and $I_0 = I_\emptyset \cap B_\emptyset^\complement$, both having strictly positive $\eta^+$-measure. 
Assume inductively that for some finite binary string $s$, the set $I_s$ has been constructed such that $\eta^+(I_s) > 0$. The restricted measure $\eta^+\vert_{I_s}$ is a non-atomic, nonnegative measure with strictly positive mass. By the contrapositive of \autoref{lem:atomforce}, there exists a ball $B_s \in \textbf{B}$ such that \[ 0 < \eta^+(I_s \cap B_s) < \eta^+(I_s). \] This defines the children $I_{s^\frown 1} = I_s \cap B_s$ and $I_{s^\frown 0} = I_s \cap B_s^\complement$, both of which are non-empty and possess strictly positive measure. By induction, this process continues indefinitely, constructing the families of subsets $(I_s)_{s \in 2^{<\mb{N}}}$ and open balls $(B_s)_{s \in 2^{<\mb{N}}}$ satisfying the conditions of \autoref{def:dyadicallyshat}. The other direction is immediate. \end{proof}

From the proof of \autoref{thm:dyadicallyshatter} it is clear that the separable metric space must have cardinality of the continuum,  as it has to support a diffuse measure. Classical conditions ensuring the ball property—such as finite doubling dimension, the Besicovitch covering property, or the existence of a uniform measure \citep{christensen1980}—provide control by imposing metric packing bounds, limiting local volume growth, or enforcing spatial homogeneity \citep[see e.g.,][]{federer1969,heinonen}. \autoref{thm:dyadicallyshatter} introduces a distinct condition based on the topological and geometric richness of the space: the collection of balls has  sufficient combinatorial capacity to allow recursive spatial partitioning without forcing every annihilating signed measure to identically vanish. If the ball property fails, all of these must fail, indicating at least intuitively that such a space lacks the standard geometric and combinatorial regularity typically relied upon in metric geometry, probability and statistics.
\begin{Corollary}\label{cor:ballfail}
    Suppose the ball property fails on the separable metric space $(S,d_S)$. Then the following hold true: \begin{enumerate}[label=(\roman*)]
        \item The metric space exhibits charged dyadic shattering;
        \item The metric space is not directionally limited and thus does not satisfy a Besicovitch-type covering property; 
        \item The metric space does not have finite doubling dimension; 
        \item The metric space cannot support a uniform measure;
        \item The metric space does not isometrically embed into a Hilbert space.
    \end{enumerate}
\end{Corollary}

Condition (i) provides a necessary and sufficient characterization, and therefore it logically implies (ii)–(v). Furthermore, (v) follows from a recent result in \citep{vD26}, which effectively demonstrates that Hilbert spaces have too rigid a geometry to `hide' Davies-type counterexamples. This resolves an open problem first raised in \citep{studeny1983} regarding whether such pathological constructions could be isometrically embedded into a Hilbert space---a question that ensued from the fact that Studen{\'y} showed they can be isometrically embedded into compact metric groups that do possess the ball property. 

It is important to emphasize that the ball property can still hold on spaces exhibiting all of the latter four negative traits. Indeed, infinite-dimensional Banach spaces generally have high combinatorial capacity (e.g., possessing both infinite VC dimension for the collection of balls and topological dyadic shattering), lack the Besicovitch covering property, do not possess finite doubling dimension, and cannot support a uniform measure (due to the non-compactness of their unit ball), yet they still satisfy the ball property. We investigate how this persists in the bigger picture via a necessary and sufficient measure-theoretic condition in the next subsection.

\begin{Remark}[Relation to empirical process theory and Boolean $\sigma$-independent sequences]\label{rem:vanhandel}
\autoref{thm:dyadicallyshatter} establishes that the failure of the ball property guarantees the existence of an integer $n^\star$ such that $VC^\star(\textbf{B}) \ge n^\star$. By Assouad's theorem \cite{assouad1983}, this implies the lower bound $VC(\textbf{B}) \ge \flo{\log_2(n^\star)}$. A natural question is whether this dual shattering can be extended indefinitely to conclude that $VC(\textbf{B})=\infty$. Generally, failure of the ball property only guarantees that there exists a specific collection of $n^\star$ balls such that all Boolean atoms they generate have nonzero signed measure. 
As seen from the proof, let $I_{\emptyset}$ be such an atom. To construct a collection of open balls $\tilde{{\textbf{B}}}=\{ \tilde{B}_1, \ldots, \tilde{B}_m\}$ such that all $2^m$ generated atoms carry nonzero measure, we require that for each step $m\ge 2$, there exists a single ball $\tilde{B}_m \in \textbf{B}$ that simultaneously bisects all $2^{m-1}$ sub-atoms generated by $\{ \tilde{B}_1, \ldots, \tilde{B}_{m-1}\}$ inside $I_\emptyset$. In empirical process theory, the existence of such a global infinite sequence of sets leads to a corresponding \textit{Boolean $\sigma$-independent sequence} \cite{vanHandel2013}, which characterizes the failure  of the universal Glivenko-Cantelli property. If the collection of indicator functions of the balls contains such a sequence then it is well-known that $VC(\textbf{B})=\infty$.  Because the dyadic  shattering allows branch-dependent choices of balls rather than enforcing a single global sequence, the failure of the ball property does not imply the existence of such a Boolean $\sigma$-independent sequence.

Consequently, working under the assumption of the ball property can be seen as a more permissive framework than standard Glivenko-Cantelli requirements. Indeed, infinite-dimensional separable Banach spaces satisfy $\text{VC}(\textbf{B}) = \infty$ (and thus fail the Uniform Glivenko-Cantelli property), yet they still preserve the ball property. This suggests that $\text{VC}(\textbf{B})  <\infty$ is too restrictive a benchmark for metric measure-theoretic regularity. However, weak convergence uniformly over a given function class is often derived via methods from empirical process theory \cite{vdVW,ArconesYu,GineZinn}.  Traditionally, a key assumption is the requirement that the functions form a VC-subgraph class. If these functions are distance-based, this effectively boils down to the constraint that $VC({\textbf B})<\infty$. However, as demonstrated in earlier work \cite{vD26}---in which limit theory was developed for nonstationary processes  taking values in a general metric space that possesses the ball property---such assumptions are technically redundant. Indeed, in that framework FCLTs for corresponding nonstationary random fields were obtained directly from first-principles fidi convergence and relative compactness.
\end{Remark}

\subsection{Characterization via $\lambda$-systems}

Consider again a separable metric space $(S,d_S)$ and denote by $\mc{M}(S)$ the space of finite signed Borel measures on $(S,\mc{B}(S))$ with $\eta(S)=0$. Observe that for any $\eta \in \mc{M}(S)$, the set 
\[
\mathcal{L}_\eta = \{A \in \mc{B}(S): \eta(A)=0\}
\]
forms a $\lambda$-system. 
Clearly, if this $\lambda$-system contains all finite intersections of the collection of open balls, then the ball property must hold since this collection forms a $\pi$-system that generates the Borel $\sigma$-algebra. However, we know from \autoref{rem:vanhandel} that even spaces with $VC(\textbf{B})<\infty$ can permit intersections to escape $\mathcal{L}_\eta$: if $VC^\star(\textbf{B})< n$ then $n$ cannot be the minimal order at which an intersection of balls lies outside $\mathcal{L}_\eta$, but it does not force all intersections to lie within the system.  All separable Banach spaces are determined by their values on balls, yet their collection of balls can have arbitrarily high combinatorial complexity (e.g., $VC(\textbf{B})=\infty$), and generally fails to form a $\pi$-system. To understand what happens, consider any $\eta \in \mc{M}(S)$ for which the collection ${\bf B}$ lies in $\mc{L}_\eta$. Then the monotone class generated by the collection of balls, denoted by $\mc{M}({\bf B})$, satisfies $\mc{M}({\bf B}) \subseteq \mc{L}_\eta$. Moreover, the smallest $\lambda$-system generated by the balls, $\lambda({\bf B})$, satisfies $\lambda({\bf B}) \subseteq \mc{L}_\eta$. We straightforwardly obtain the following equivalence:

\begin{Proposition}\label{prop:lambdasepc}
    The ball property holds on a separable metric space $(S,d_S)$ if and only if $\lambda({\bf B})$ is a separating class of sets. 
\end{Proposition}

In fact, this establishes that the ball property does not require $\lambda({\bf B})=\mc{B}(S)$. It suffices that $\lambda({\bf B})$ contains a separating subclass, i.e., that the space supports a separating class that can be realized from open balls via standard $\lambda$-system operations, such as the countable monotone limit of nested open balls or the countable disjoint union of open balls. 

This characterization unifies classical results in the literature. Early investigations by Hoffmann-J{\o}rgensen \citep{hoffmann1975}, Christensen \citep{christensen1980}, and Riss \citep{riss1989} established the ball property under specialized geometric settings—such as spaces with uniform measures, Hilbert spaces, or vector spaces with smooth norms. This line of inquiry culminated in the general theorem of Preiss and Ti{\v{s}}er \citep{PreissTisier91}, proving that all separable Banach spaces possess the ball property. A closer inspection of their proof reveals that the key argument relies precisely on constructing open convex cones via countable, monotonic unions of expanding open balls (e.g., $D = \bigcup_{n=1}^\infty B(nx,n)$, where $\|x\|=1$) combined with Zorn's lemma to essentially establish the separating property for the collection of half-spaces. 
Through the lens of \autoref{prop:lambdasepc}, their result can be understood as establishing that the normed linear structure of any separable Banach space is sufficient to force $\lambda({\bf B})$ to be a separating class, regardless of the VC dimension of the balls.

It is important to recognize that half-spaces of the form $\{x \in E: f(x) <c\}$, where $c \in \mb{R}$ and $f \in E^\prime$ is a continuous linear functional on $E$, form a generating subclass of $1$-dimensional cylinder sets for $\mc{B}(S)$. We will push this realization further in the next section to show when outer measures constructed on such sets decouple from those constructed via the balls.

\section{Product spaces}\label{sec:sec3}

Consider a product space given by $\Omega = \prod_{n=1}^\infty E_n$, where each $E_n$ is a Polish space. We endow $\Omega$ with the product topology. It is well-known that the product $\sigma$-algebra $\F:= \otimes_{n=1}^\infty \mc{B}(E_n)$ coincides with the Borel $\sigma$-algebra on $\Omega$. 
Consider the coordinate mappings
\[
Y_ n: \Omega \to E_n, \quad Y_n(\omega) = \omega(n),
\]
and recall that two finite Borel measures $\mu, \nu$ on $(\Omega, \F)$ are equal if and only if 
\[
\mu(C)= \nu(C) 
\]
for all cylinder sets $C$. That is, they must agree on any set $C$ of the form $
 \bigcap_{i=1}^n Y^{-1}_{i}(A_{i})$
where  $A_{i} \in \mc{B}(E_{i})$ for $i =1, \ldots, n$, $n \in \mb{N}$. Denote the collection of cylinder sets by $\mc{C}$. 
\begin{lemma}\label{lem:cylininlamb}
Let $\Omega$ be the product space defined above. $\lambda({\textbf B})$ is separating if $\mc{C}  \subseteq \lambda({\textbf B})$.
    
\end{lemma}
\begin{proof}
Suppose $\mc{C}  \subseteq \lambda({\textbf B})$. Since $\mc{C}$ is a $\pi$-system that generates the Borel sigma algebra, we have $\lambda(\mc{C}) =\mc{B}(\Omega)$. Therefore, if $\mc{C} \subseteq \lambda({\textbf B})$ then $\mc{B}(\Omega) \subseteq \lambda({\textbf B})$.
\end{proof}

The following is an immediate adaptation of  \autoref{cor:ballfail} to product spaces.

\begin{Corollary} \label{cor:ballfailg}
    Let $\Omega$ be a Polish space arising as a countable product of Polish spaces, and let $\rho$ be a complete metric on $\Omega$ that metrizes the product topology. Suppose the ball property fails. Then the following must hold: \begin{enumerate}[label=(\roman*)]
        \item The metric space exhibits charged dyadic shattering;
        \item The metric space is not directionally limited and thus does not satisfy a Besicovitch-type covering property; 
        \item The metric space does not have finite doubling dimension; 
        \item The metric space cannot support a uniform measure; 
         \item The metric space does not isometrically embed into a Hilbert space;
        \item The collection of cylinder sets satisfies $\mc{C} \not\subseteq {\lambda}({\textbf B})$.
    \end{enumerate}
\end{Corollary}

The characterization of measures on a product space via cylinder sets actually provides a  key insight into the ball property, and it is what Davies exploited in his counterexample (discussed below). Consider a collection of Borel probability measures $\{\mu_n\}$ where $\mu_n$ is defined on $(X_n, \mc{B}(X_n))$ with $X_n: = \prod_{k=1}^n E_k$. If this collection satisfies the projective consistency condition, Kolmogorov's Extension Theorem ensures the existence of a unique probability measure $\mu$ on $(\Omega,\mc{B}(\Omega))$ such that $\{\mu_n\}$ are the corresponding finite-dimensional distributions.
   Generally, this argument is proven as follows: the collection of cylinders $\mathcal{C} =\bigcup_{n=1}^\infty \mathcal{C}_n$, where $\mathcal{C}_n$ consists of sets of the form
\[C_n := A \times \prod_{i=n+1}^\infty E_i \quad \text{for } A \in \mathcal{B}(X_n),\]
forms an algebra. Projective consistency allows us to define a finitely additive set function $\mu_0$ on this algebra of cylinders. By the inner regularity of Borel measures (which holds on Polish spaces by Ulam's theorem) and Carath{\'e}odory's extension theorem, $\mu_0$ is extended to a countably additive set function on $(\Omega,\mathcal{B}(\Omega))$ via the cylinder outer measure, which is given by
\[\mu^\star_{\mathcal{C}}(B) = \inf\Big\{\sum_{i=1}^\infty \mu_0(C_i) : B \subseteq \bigcup_{i=1}^\infty C_i, \; C_i \in \mathcal{C}\Big\}.\]
Under the conditions of Kolmogorov's extension theorem, the restriction of $\mu^\star_{\mathcal{C}}$ to the Borel $\sigma$-algebra $\mathcal{B}(\Omega)$ yields the unique countably additive probability measure $\mu$. Therefore, for any Borel set $B$, we have $\mu_{\mathcal{C}}^\star(B)=\mu(B)$. 
In contrast, consider the metric outer measure defined by centered ball covers 
\[\mu_{\textbf{B}}^\star(E) = \sup_{\epsilon>0} \inf\Big\{\sum_{k=1}^\infty \mu(B_k) : E \subseteq \bigcup_{k=1}^\infty B_k, \; B_k \in \textbf{B}_\epsilon(E)\Big\},\]
where $\textbf{B}_\epsilon(E) = \{B(x,r) : x \in E, r<\epsilon\}$ is the collection of open balls centered on the set $E$, and $\mu$ is the given finite Borel measure on $(\Omega,\mathcal{B}(\Omega))$. 

\begin{Remark}
It is important to emphasize here that the metric outer measure $\mu_{\textbf{B}}^\star$ is defined using covers whose centers lie in the set $E$. In standard geometric measure theory on finite-dimensional spaces, metric outer measures are often defined using uncentered balls as the two definitions become equivalent. To transition from an uncentered cover to a centered cover, one typically shifts the centers to points within the set, which requires expanding the radii by a factor of $2$ via the triangle inequality. 

In finite dimensions, the sum of the measures of these expanded balls is controlled either by a doubling property of the measure (i.e., $\mu(B(x, 2r)) \le C \mu(B(x,r))$) or by the Besicovitch Covering Property, which bounds the multiplicity of overlapping balls. However, both the doubling property and the Besicovitch property fail in infinite-dimensional  spaces. Therefore, in infinite-dimensional Polish spaces, the uncentered and centered metric outer measures are not guaranteed to be equivalent. Uncentered covers allow centers to sit arbitrarily far outside of the set $E$, allowing for a tighter cover than  can be obtained via the centered metric outer measure since  $\mu^\star_{\text{uncentered}}(E) \le \mu^\star_{\text{centered}}(E)$. Conceptually analogous to the projective consistency condition required by Kolmogorov's extension theorem,  a consistent notion of metric outer measure is one that remains intact when projecting down to marginals or lifting to infinite products. This consistency breaks when the centers are decoupled from the sets they approximate, making the centered metric outer measure the appropriate geometric object to characterize the ball property.
\end{Remark}

Ulam's theorem guarantees that every finite Borel measure on a Polish space is inner regular with respect to compact subsets. This leads to the following equivalence:
\begin{Proposition}\label{prop:cylbalouter}
Let $\Omega = \prod_{k=1}^\infty E_k$ be a Polish product space equipped with a complete metric $\rho$ metrizing the product topology. Suppose $\mu$ is a Borel probability measure on $(\Omega,\mc{B}(\Omega))$ constructed from finite-dimensional distributions via Kolmogorov's Extension Theorem. Then $(\Omega, \rho)$ has the ball property with respect to $\mu$ if and only if for any compact set $K\subseteq \Omega$, we have 
\[\mu_{\textbf{B}}^\star(K) \le \mu_{\mathcal{C}}^\star(K). \tageq \label{eq:outercb}\]
\end{Proposition}

 We can relate the cylinder outer measure to the finite-dimensional projection by recalling that any compact set $K \subseteq \Omega$ can be represented as an intersection of a decreasing sequence of cylinders based on the projections. More specifically, $K = \bigcap_{n=1}^\infty \pi_n^{-1}\big(\pi_n(K)\big)$ where each $C_n:=\pi_n^{-1}\big(\pi_n(K)\big)$ is a cylinder set with an $n$-dimensional base. Under the conditions of the above statement, continuity from above yields \[\mu^\star_{\mc{C}}(K)= \mu(K) = \lim_{n\to \infty }\mu(\pi_n^{-1}\big(\pi_n(K)\big))= \lim_{n\to \infty}\mu_n(\pi_n(K)). \tageq \label{eq:cylindlim}\]
This connects the right-hand side of \eqref{eq:outercb} directly to the finite-dimensional measures $\mu_n$. 
To relate the metric outer measure to projective spaces $X_n$, we must formalize their metric geometry. To make this compatible with the metric $\rho$ on $\Omega$, we consider an induced metric on $X_n$ via  canonical embedding. Fix a sequence $e^\star=(e^\star_1, e^\star_2, \ldots) \in \Omega$ and define the embedding $i_n: X_n \to \Omega$ by 
\[
i_n(x_n) =(e_1,\ldots,e_n, e^\star_{n+1},e^\star_{n+2}, \ldots),  
\]
where $x_n= (e_1,\ldots,e_n) \in X_n$. We then equip $X_n$ with the induced metric $d_{X_n}(x_n,y_n):=\rho(i_n(x_n), i_n(y_n))$. Since each $X_n$ is a finite product of Polish spaces, it is Polish. Furthermore $i_n(X_n)= \prod_{k=1}^n E_k \times \{e^\star_{n+1}\} \times \{e^\star_{n+2}\}\times \cdots$ is 
a closed subspace of $\Omega$, 
 implying that the metric $d_{X_n}$ metrizes the product topology on $X_n$. Since $\rho$ is complete so is $d_{X_n}$, making each $(X_n,d_{X_n})$ a well-defined complete separable metric space. We then denote by $\mu^\star_{n,\textbf{B}}$ the metric outer measure on $(X_n,d_{X_n})$ constructed using $\mu_n$.  

For the inequality in \eqref{eq:outercb}, it is not sufficient that the ball property holds on each finite-dimensional projection individually, i.e., that $\mu^\star_{n,{\textbf B}}=\mu_n$ for each $n\in \mb{N}$. To preserve this property under limits, the metric outer approximations must not deteriorate as $n \to \infty$. For this, we must constrain the number of open balls required to achieve a tight enough approximation on the projected compact sets. 
\begin{definition}
    For a compact set $K_n \subseteq X_n$, radius $\epsilon > 0$, and tolerance $\delta > 0$, define the \textit{projection covering number} $M_n(K_n, \epsilon, \delta)$ as the smallest integer $M \in \mb{N}$ such that there exists a finite collection of open balls $\{B_{X_n}(x_j, r_j)\}_{j=1}^{M}$ of radius $r_j \le \epsilon$ with centers $x_j \in K_n$ satisfying 
\begin{align*}
    K_n \subseteq \bigcup_{j=1}^{M} B_{X_n}(x_j, r_j),
\end{align*}
and 
\begin{align*}
    \sum_{j=1}^{M} \mu_n\big(B_{X_n}(x_j, r_j)\big) \le \mu_n(K_n) + \delta.
\end{align*}
If no such integer $M$ exists, then we set $M_n(K_n, \epsilon, \delta) = \infty$.
\end{definition}

The following theorem entirely characterizes the ball property on infinite product spaces.

\begin{thm}\label{thm:ballprod}
Let $\Omega = \prod_{k=1}^\infty E_k$ be a Polish product space equipped with a complete metric $\rho$ metrizing the product topology, and let $\mu$ be a finite Borel measure on $\Omega$. For every $n \in \mb{N}$, let $\mu_n =\mu \circ \pi^{-1}_n$ be the projected measure on $(X_n, d_{X_n})$. 
The following are equivalent:
\begin{enumerate}[label=(\roman*)]
    \item $(\Omega,\rho)$ satisfies the ball property with respect to $\mu$, and for every $n \in \mb{N}$, $(X_n, d_{X_n})$ satisfies the ball property with respect to $\mu_n$.
    \item For every compact set $K \subseteq \Omega$, radius $\epsilon > 0$, and $\delta > 0$, the projection covering numbers are uniformly bounded:
    \begin{equation*}
        \sup_{n \in \mb{N}} M_n(\pi_n(K), \epsilon, \delta) < \infty.
    \end{equation*}
\end{enumerate}
\end{thm}

The conditions of \autoref{thm:ballprod} give us a precise characterization of when a finite Borel measure on a product space---constructed via Kolmogorov's Extension Theorem---is fully determined by its values on balls. However, condition (ii) reveals that this inherently depends on the intricate interplay between the measure and the corresponding metric geometries on the spaces $X_n$, $n\in \mb{N}$. If no uniform bound on the measure-theoretic covering numbers exists, then from \autoref{prop:cylbalouter}, the metric outer approximations of compact sets decouple from the cylinder outer approximations; that is, for some compact set $K \subseteq \Omega$, 
\[
\mu^\star_{\textbf{B}}(K) > \mu^\star_{\mc{C}}(K) = \mu(K). \tageq\label{eq:muBallCyl}
\]
This decoupling demonstrates that the open balls fail to approximate compact sets without excess measure, implying that the linear span of the indicator functions of open balls is not dense in  $L^1(\Omega, \mc{B}(\Omega), \mu)$. Tying this to the functional-analytic characterization of the ball property as given in \autoref{lemm:ballsspan}, this guarantees in effect the existence of a non-zero signed Borel measure $\eta$ on $(\Omega, \mc{B}(\Omega))$ that annihilates all balls. We show in the next section how such a non-zero signed Borel measure is explicitly constructed on a product metric space. 

\begin{proof}[Proof of \autoref{thm:ballprod}]

We first show the implication $\text{(ii)} \implies \text{(i)}$.  Fix a compact set $K \subseteq \Omega$, an $\epsilon > 0$, and a $\delta > 0$. Let $K_n = \pi_n(K) \subseteq X_n$. By  (ii), there exists a uniform bound $M \in \mb{N}$ independent of $n$ such that we can cover $\pi_n(K)$ using at most $M_n \le M$ open balls in $X_n$, i.e.,
\[
\pi_n(K) \subseteq \bigcup_{j=1}^{M_n} B_{X_n}(x_{n,j}, r_{n,j}) \tageq\label{eq:Kncover}
\]
where $r_{n,j} \le\epsilon$, and  $x_{n,j} \in \pi_n(K)$, such that
\[
\sum_{j=1}^{M_n} \mu_n\big(B_{X_n}(x_{n,j}, r_{n,j})\big) \le \mu_n(\pi_n(K)) + \delta.
\]
Since this holds for any $\delta$, taking $\delta \downarrow 0$ shows that $\mu^\star_{n,\textbf{B}}(\pi_n(K)) \le \mu_n(\pi_n(K))$. In fact, this holds for every compact set $\mc{K}_n \subseteq X_n$ as we can lift this to a compact set in $\Omega$ via the canonical embedding $i_n$. 
Hence, $\mu^\star_{n,\textbf{B}}(\mc{K}_n) = \mu_n(\mc{K}_n)$ for any compact set $\mc{K}_n\subseteq X_n$ and thus the Polish space $(X_n,d_{X_n})$ satisfies the ball property with respect to $\mu_n$ for all $n \in \mb{N}$.

To show that the measure $\mu$ on $(\Omega, \mc{B}(\Omega))$ is uniquely determined by its values on open balls, it suffices again (since $\Omega$ is Polish) to show that for any compact set $K \subseteq \Omega$ and any $\delta > 0$, we have $\mu_{\textbf{B}}^\star(K) \le \mu(K) + \delta$. To this end, we lift the cover  $\{B_{X_n}(x_{n,j}, r_{n,j})\}_{j=1}^{M_n}$ to $\Omega$ via the canonical embedding $i_n$. Define the centers $\omega_{n,j} = i_n(x_{n,j}) \in \Omega$. Since $i_n: X_n \to \Omega$ is an isometric embedding by the definition of $d_{X_n}$, the preimage of an open ball in $\Omega$ under $i_n$ is exactly the corresponding open ball in $X_n$:
\[
i_n^{-1}\big(B_\Omega(\omega_{n,j}, r_{n,j})\big) = B_{X_n}(x_{n,j}, r_{n,j}). \tageq\label{eq:inomegatoxn}
\]
Indeed, if $y_n \in B_{X_n}(x_{n,j}, r_{n,j})$ then $d_{X_n}(x_{n,j},y_n)=\rho(\omega_{n,j},i_n(y_n)) < r_{n,j}$ so that $i_n(y_n) \in B_\Omega(\omega_{n,j}, r_{n,j})$. This completely characterizes the preimage because $i_n^{-1}(B) = \{y_n \in X_n: i_n(y_n) \in B\}$. 
Hence, applying $i_n$ to both sides and using \eqref{eq:Kncover}, we get that 
\begin{align*}
i_n(K_n) \subseteq \bigcup_{j=1}^{M_n} i_n(B_{X_n}(x_{n,j}, r_{n,j})) \subseteq \bigcup_{j=1}^{M_n} B_\Omega(\omega_{n,j}, r_{n,j}).
\end{align*}
Furthermore,
\[
(i_n)_*\mu_n(B_\Omega(\omega_{n,j}, r_{n,j})) = \mu_n\big(i_n^{-1} (B_\Omega(\omega_{n,j}, r_{n,j}))\big) = \mu_n\big( B_{X_n}(x_{n,j}, r_{n,j})\big).
\]
Since $\rho$ metrizes the product topology, the sequence of compact sets $i_n(K_n)$ converges to $K$ in the Hausdorff metric as $n \to \infty$. By the definition of the projection covering number, the covers are centered on the sets themselves, meaning $x_{n,j} \in K_n$. Consequently, the lifted centers $\omega_{n,j} = i_n(x_{n,j})$ strictly belong to $i_n(K_n)$. In any metric space, the union of a Hausdorff-convergent sequence of compact sets and its limit forms a compact set \citep[see e.g.,][Lemma 4.6]{Bargetzetc2023}. Therefore, the set $\mathcal{S} = K \cup \bigcup_{n=1}^\infty i_n(K_n)$ is a compact subset of $\Omega$. Since $M_n \in \{1,\ldots, M\}$ for all $n$, the pigeonhole principle yields a subsequence $\{n_k\}_k$ and a constant integer $M^\star\le M$ such that $M_{n_k}=M^\star$ for all $k$. The sequence of tuples $\{(\omega_{n_k,1}, r_{n_k,1}, \ldots, \omega_{n_k,M^\star}, r_{n_k,M^\star})\}_k$ then lives in a compact space $(\mc{S} \times [0,\epsilon])^{M^{\star}}$. 
By sequential compactness, we can pass to a further subsequence such that for each $j \in \{1, \dots, M^\star\}$, we have $(\omega_{n_{k_l},j}, r_{n_{k_l},j}) \to (\omega_j, r_j)$ where $\omega_j \in K$ and $r_j\le  \epsilon$.

 In order to establish weak convergence  $\mu_n\circ i_n^{-1} \Rightarrow \mu$ on $(\Omega, \mc{B}(\Omega))$, denote by $\nu_n=\mu_n\circ i_n^{-1}$. To show that $\nu_n \circ \pi_k^{-1} \Rightarrow \mu \circ \pi_k^{-1}$, for all $k \in \mb{N}$, consider a Borel measurable set $A=\prod_{i=1}^k A_i \in \mc{B}(\prod_{i=1}^k E_i)$ and  observe that
    \begin{align*}
        \nu_n \circ \pi_k^{-1}(A)=\mu_n\Big(\{x_n \in X_n: i_n(x_n) \in \pi_k^{-1}(A)\}\Big).
    \end{align*}
By definition, $i_n(x_n)=(e_1, \ldots, e_n, e^\star_{n+1}, \ldots)$ and thus we obtain for $k>n$,
\begin{align*}   
\nu_n \circ \pi_k^{-1}(A)& =\mu\Big( \bigcap_{i=1}^n \{e \in \Omega: e_i\in A_i\} \cap \bigcap_{i=n+1}^{k} \{e \in \Omega: e^\star_i\in A_i\}\cap  \bigcap_{i=k+1}^\infty \{e \in \Omega: e_i^\star\in E_i\}\Big)\\&=\mu_n(\prod_{i=1}^n A_i) \delta_{(e_{n+1}^\star,\ldots, e^\star_k)}(\prod_{i=n+1}^k A_i),
\end{align*}
whereas for $k\le n$,
\begin{align*}   
\nu_n \circ \pi_k^{-1}(A) =\mu\Big( \bigcap_{i=1}^{k} \{e \in \Omega: e_i\in A_i\} \Big)=\mu \circ \pi_k^{-1}(\prod_{i=1}^k A_i). 
\end{align*}
Hence, because $k$ is fixed it follows that  $\nu_n \circ \pi_k^{-1} \Rightarrow \mu \circ \pi_k^{-1}$, for all $k \in \mb{N}$. Denote by $\nu_{n,\{k\}}= \nu_n \circ \pi_{\{k\}}^{-1}$ the $k$th one-dimensional marginal of $\nu_n$. By construction for $n<k$, $\nu_{n,\{k\}}= \delta_{e^\star_k}$, while for $n \ge k$, $\nu_{n,\{k\}}=\mu_{\{k\}}$ by projective consistency. Hence for fixed $k$, the sequence $\{ \nu_{n,\{k\}}\}_{n \in \mb{N}}$ on $(E_k, \mc{B}(E_k))$ entirely  consists of the measures $\delta_{e^\star_k}$ and $\mu_{\{k\}}$. Without loss of generality, assume that $\mu$ is a Borel probability measure on the Polish space $(\Omega, \mc{B}(\Omega))$ (otherwise, we can simply normalize by $\mu(\Omega)$), its $k$th one dimensional marginal $\mu_{\{k\}}$ is a tight probability measure on $(E_k, \mc{B}(E_k))$. Consequently there exists a compact set $\mc{K}_k \subseteq E_k$ such that 
\[\mu_{\{k\}}(\mc{K}_k) >1-\frac{\epsilon}{2^k}.\] Hence, 
 define the set $V=\prod_{k=1}^\infty V_k$ where $V_k=\mc{K}_k \cup \{e^\star_k\}$. Then $V \subseteq \Omega$ is compact,
 and by a union bound
 \begin{align*}
    \nu_n(V^\complement)&=\mu\Big(\big\{e \in \Omega: i_n(\pi_n(e)) \in  \bigcup_{k=1}^\infty \{y \in \Omega: y_k \notin V_k\}\big\}\Big)
    \\&=\mu\Big(\bigcup_{k=1}^\infty\big\{e \in \Omega: \text{$k$th coordinate of $i_n(\pi_n(e))$ does not lie in  $V_k$}\big\}\Big)
    \\& \le \sum_{k=1}^\infty \mu\Big(\big\{e\in \Omega: \text{$k$th coordinate of $i_n(\pi_n(e))$ does not lie in  $V_k$}\}\Big). 
 \end{align*}
 Observe that if $k>n$, then the $k$th coordinate of $i_n(\pi_n(e))$ is $e^\star_k$, whereas if $k\le n$, it is $e_k$. Since $\mu(\{e \in \Omega: e_k \notin V_k\})=\mu(\{e \in \Omega: e_k \notin \mc{K}_k\})\le \epsilon/2^k$ and $\mu(\{e \in \Omega: e_k^\star \notin V_k\})=0$, we obtain that 
 \[
 \nu_n(V^\complement) \le \sum_{k=1}^\infty \frac{\epsilon}{2^k} = \epsilon. 
\]
Consequently, $
\inf_{n \in \mb{N}} \nu_n(V) >1-\epsilon$, showing the sequence $\{\nu_n\}$ on $(\Omega, \mc{B}(\Omega))$ is uniformly tight. 
This establishes the weak convergence $(i_n)_*\mu_n \Rightarrow \mu$ on $(\Omega, \mc{B}(\Omega))$. 

Returning to the subsequence of indices $\{n_{k_l}\}_{l \in \mb{N}}$ from the sequential compactness argument, observe that $d_H(i_{n_{k_l}}(K_{n_{k_l}}), K) \to 0$ as $l \to \infty$ implies that $K \subseteq  \bigcup_{j=1}^{M^\star} \bar{B}_\Omega(\omega_j, r_j)$ where $\bar{B}(\omega_j, r_j)$ denotes the closed ball centered at $\omega_j$ with  radius $r_j$. Thus, for any $\epsilon^\prime>0$, the open balls $\{{B}_\Omega(\omega_j, r_j+\epsilon^\prime)\}_{j=1}^{M^\star}$ yield a valid open cover. Since $\mu$ is finite, the boundary $\partial B_{\Omega}(\omega_j,r)$ has positive measure for at most countably many radii $r>0$, so we may assume  without loss of generality that $\mu(\partial B_{\Omega}(\omega_j,r_j))=0$ for all $1 \le j \le M^\star$. By definition of the metric outer measure and continuity of measure from above, we obtain
\[
\mu_{\textbf{B}}^\star(K) \le \lim_{\epsilon^\prime \downarrow 0}\sum_{j=1}^{M^\star} \mu\big(B_\Omega(\omega_j, r_j+\epsilon^\prime)\big)=\sum_{j=1}^{M^\star} \mu\big(\bar{B}_\Omega(\omega_j, r_j)\big) =\sum_{j=1}^{M^\star} \mu\big({B}_\Omega(\omega_j, r_j)\big).
\]
Now, for any $\eta > 0$, choose $l_0$ sufficiently large such that for all $l \ge l_0$, and all $1 \le j \le M^\star$, we have the inclusion $B_\Omega(\omega_j, r_j - \eta) \subseteq B_\Omega(\omega_{n_{k_l},j}, r_{n_{k_l},j})$. Applying continuity of measure from below, and the Portmanteau theorem for open sets on the lower limit, we obtain as a bound on the outer measure of $K$:
\begin{align*}
    \mu_{\textbf{B}}^\star(K) &\le \sum_{j=1}^{M^\star} \mu\big(B_\Omega(\omega_j, r_j)\big) \\
    &= \lim_{\eta \downarrow 0}  \sum_{j=1}^{M^\star} \mu\big(B_\Omega(\omega_j, r_j - \eta)\big) \\
    &\le   \lim_{\eta \downarrow 0} \sum_{j=1}^{M^\star}  \liminf_{l \to \infty}\nu_{n_{k_l}}\big(B_\Omega(\omega_j, r_j - \eta)\big) \\
     &\le  \lim_{\eta \downarrow 0} \sum_{j=1}^{M^\star}   \sup_{l\ge 1} \inf_{m \ge l}\nu_{n_{k_m}}\big(B_\Omega(\omega_j, r_j - \eta)\big) \\
      &\le  \lim_{\eta \downarrow 0} \sum_{j=1}^{M^\star}   \sup_{l\ge l_0} \inf_{m \ge l}\nu_{n_{k_m}}\big(B_\Omega(\omega_{n_{k_m},j}, r_{n_{k_m},j})\big) \\
    &\le \liminf_{l \to \infty} \sum_{j=1}^{M^\star} \mu_{n_{k_l}} \big(i_{n_{k_l}}^{-1}(B_\Omega(\omega_{n_{k_l},j}, r_{n_{k_l},j}))\big) \\
    &\le \limsup_{l \to \infty} \Big( \mu_{n_{k_l}}(K_{n_{k_l}}) + \delta \Big),
\end{align*}
where we used \eqref{eq:inomegatoxn}. By equation \eqref{eq:cylindlim}, $\lim_{n\to\infty} \mu_n(K_n) = \mu(K)$, and therefore 
\[
\mu_{\textbf{B}}^\star(K) \le \mu(K) + \delta.
\]
Since $\delta > 0$ was arbitrary, taking $\delta \downarrow 0$ yields $\mu_{\textbf{B}}^\star(K) \le \mu(K)$, completing the proof of the implication $\text{(ii)} \implies \text{(i)}$.

Now, assume (i) holds. We will show that (ii) must hold. Fix a compact set $K \subseteq \Omega$, $\epsilon>0$ and $\delta>0$. Since the ball property holds,  $\mu_{\textbf{B}}^\star(K) =\mu(K)$. Hence, by definition of the centered metric outer measure and compactness of $K$ we can find a finite collection of balls $\{B_\Omega(x_j,r_j)\}_{j=1}^M$ with radii $r_j \le \epsilon$ and centers $x_j \in K$ that covers $K$ such that 
\[
\sum_{j=1}^M \mu(B_\Omega(x_j,r_j)) \le \mu(K)+\delta/2.
\]
For each fixed $1 \le j \le M$, consider then the sequence of cylinders of the form $C_{n,j} = \pi_n^{-1}\big(B_{X_n}(\pi_n(x_j),r_j)\big)$.  Note that $y \in C_{n,j} = \{z\in \Omega: \pi_n(z) \in \big(B_{X_n}(\pi_n(x_j),r_j)\big)\}$, which implies that 
\[
d_{X_n}(\pi_n(y), \pi_n(x_j)) = \rho(i_n(\pi_n(y)), i_n(\pi_n(x_j)))< r_j.
\]
Therefore, if $y \in \limsup_{n\to \infty} C_{n,j}$, we can find a strictly increasing subsequence $\{n_k\}_{k \in \mb{N}}$ such that $y \in C_{n_k,j}$ for all $k \in \mb{N}$. By continuity of $\rho$ with respect to the product topology and because $\rho(i_n(\pi_n(y)),y) \to 0$ as $n \to \infty$, we find
\[
\lim_{k \to \infty}\rho(i_{n_k}(\pi_{n_k}(y)), i_{n_k}(\pi_{n_k}(x_j)))=\rho(y,x_j).
\]
Since all terms are strictly  bounded from above by $r_j$, it follows that $\rho(y,x_j)\le r_j$. Consequently, $y \in \limsup_{n\to \infty} C_{n,j} \subseteq \bar{B}_{\Omega}(x_j, r_j)$ where $\bar{B}_{\Omega}(x_j, r_j)$ denotes the closed ball with radius $r_j$ centered at $x_j$. Since $\mu$ is finite, we can pick the radii $\{r_j\}_{j=1}^M$ such that $\mu(\partial B_{\Omega}(x_j,r_j))=0$. Hence, reverse Fatou's lemma for sets together with monotonicity of measure  yield
\[
\limsup_{n\to \infty} \mu(C_{n,j}) \le \mu(B_{\Omega}(x_j,r_j)).
\]
Summing over $j = 1, \dots, M$ and using $\mu_n(B_{X_n}(\pi_n(x_j),r_j)) = \mu(C_{n,j})$, we have
\[
\limsup_{n \to \infty } \sum_{j=1}^M \mu_n (B_{X_n}(\pi_n(x_j),r_j))\le  \sum_{j=1}^M \mu(B_\Omega(x_j,r_j)) \le \mu(K)+\delta/2.
\]
By \eqref{eq:cylindlim}, $\lim_{n\to \infty}\mu(\pi_n^{-1}(\pi_n(K)))=\mu(K)$, and thus there exists an integer $N$ such that for all $n \ge N$
\[
\sum_{j=1}^M \mu_n (B_{X_n}(\pi_n(x_j),r_j))\le  \mu_n(\pi_n(K))+\delta.
\]
implying that $M_n(\pi_n(K), \epsilon, \delta)\le M$ for all $n \ge N$. By assumption, $(X_n, d_{X_n})$ satisfies the ball property with respect to $\mu_n$ for every $n < N$, ensuring each $M_n(\pi_n(K), \epsilon, \delta)$ is finite. Setting $M_{\max} = \max\{M_1(\pi_1(K),\epsilon,\delta), \ldots, M_{N-1}(\pi_{N-1}(K),\epsilon,\delta), M\}$, we find 
\[
\sup_{n \in \mb{N}} M_n(\pi_n(K), \epsilon, \delta) \le  M_{\max}<\infty,
\]
which completes the proof.
\end{proof}

\subsection{The pathology of Davies' construction}

Given the structural requirements outlined in \autoref{cor:ballfail} for the ball property to fail, it is natural to ask whether the class of such spaces is non-empty. Davies \cite{davies1971} provided a definitive affirmative answer by constructing a canonical counterexample. Subsequent literature examined the underlying structural properties of his space; notably, Studen{\'y} \cite{studeny1983} (see also Darst \cite{Darst1973}) demonstrated that Davies' space is homeomorphic to the Cantor set and can be naturally embedded into a compact metric group. In what follows, we aim to highlight the exact mechanics behind his construction in light of our previous results, especially \autoref{thm:ballprod}. 
Davies constructs a product space $\Omega = \prod_{n=1}^\infty E_n$, where each coordinate space is defined as $E_n = \{(i,j): 1 \le i \le N_n, 0 \le j \le N_n\}$. The integer $N_n$ controls the local complexity and is specified later to ensure the counterexample works. The set $E_n$ is partitioned into two types of nodes:
    \begin{itemize}
        \item[-] central nodes: $V_C = \{(i,0) \mid 1 \le i \le N_n\}$, 
        \item[-] peripheral nodes: $V_P = \{(i,j) \mid 1 \le i \le N_n, 1 \le j \le N_n\}$. 
    \end{itemize}
    The geometry of $E_n$ is determined by an undirected graph $G_n$, constructed such that the central nodes $V_C$ form a fully connected clique, and each peripheral node $(i,j) \in V_P$ is connected exclusively to its corresponding central node $(i,0)$.
     Davies builds a metric on this product space of sequences $\Omega$ that metrizes the product topology and  depends on a `stopping time'  $\tau(x,y)=\inf\{n : x_n \neq y_n\}$ where $x_n$ and $y_n$ denote, respectively, the $n$th coordinates of $x,y \in \Omega$. It is defined as
    \begin{align*}
        \rho(x,y) =\begin{cases}
        0 & \text{if } x=y, \\ 
        \rho_{\tau}(x_{\tau},y_{\tau}) & \text{if } x_n = y_n \text{ for } n<\tau \text{ and } x_{\tau}\neq y_{\tau},
    \end{cases}
    \end{align*}
    where $\rho_n(x_n, y_n)$ is a custom metric on the coordinate space $E_n$. Note that this setting is perfectly compatible with the general framework established around \autoref{thm:ballprod}. In particular, recall that the metric $d_{X_n}$ on the finite product space $X_n= \prod_{i=1}^n E_i$ is defined by pulling back the metric $\rho$ from $\Omega$ via the canonical embedding $d_{X_n}(w,z) = \rho(i_n(w),i_n(z))$, where $w,z \in X_n$. Since $i_n(w)$ and $i_n(z)$ have identical tails, $\tau(i_n(w),i_n(z)) \le n$, restricting Davies' metric exactly to the first $n$ coordinates without metric distortion. Indeed, if $w=z$, then $\tau(i_n(w),i_n(z)) = \infty$ and $d_{X_n}(w,z)=0$, whereas if $w \neq z$ then $\tau(i_n(w),i_n(z)) \le n$ and $d_{X_n}(w,z) =\rho_{\tau}(w_\tau,z_\tau)$.  
    
    The custom metric on $E_n$ is then defined using the graph $G_n$ as:
    \begin{align*}
    \rho_n(x_n,y_n) =\begin{cases}  0 & \text{if $x_n=y_n$,}\\
        \frac{1}{2^{n}} &\text{if $x_n$ and $y_n$ are connected in $G_n$}, \\ 
        \frac{1}{2^{n-1}} &\text{otherwise.}
    \end{cases}
    \end{align*}
    The cylinder sets are of the form
    \[
    C_k(x) = \{ y \in \Omega: y_1=x_1, \ldots, y_k=x_k\}.
    \]
Now, consider the structure of a closed ball $\bar{B}(x,r)$ for a radius $r=1/2^k$. By the definition of $\tau$ and $\rho$, a point $y$ belongs to $\bar{B}(x, 1/2^k)$ if and only if $y$ agrees with $x$ on the first $k-1$ coordinates, and their $k$-th coordinates are either identical or connected in $G_k$. Consequently, we find
    \begin{align*}
        \bar{B}(x,1/2^k) &= \{y \in \Omega: y_1=x_1, \ldots, y_{k-1}=x_{k-1}, \text{ and } (y_k = x_k \text{ or } y_k \sim x_k \text{ in } G_k)\} \\
        &= C_k(x) \cup F_k(x), \tageq\label{eq:balliscylplusfringe}
    \end{align*}
    where $F_k(x)$ is a non-empty ``fringe'' defined by
    \[
        F_k(x) = \{y \in \Omega: y_1=x_1, \ldots, y_{k-1}=x_{k-1}, \text{ and } y_k \sim x_k \text{ in } G_k \text{ with } y_k \neq x_k\}.
    \]
 Thus, the ball $\bar{B}(x, 1/2^k)$ is a strict superset of $C_k(x)$, but a strict subset of $C_{k-1}(x)$. 
To construct the counterexample, Davies first builds  finite Borel measures that are equal on closed balls but differ on the cylinder sets, and then chooses $N_n$ to ensure that these can be extended to corresponding finite Borel measures $\mu_1$ and $\mu_2$ on $(\Omega, \mc{B}(\Omega))$ by verifying  Kolmogorov's consistency condition.  To complete the counterexample, two distinct probability measures $\mu$ and $\nu$ are then constructed on $\Omega \cup \Omega^\prime$ where $\Omega^\prime$ is a disjoint copy of $\Omega$ with the roles of $\mu_1$ and $\mu_2$ interchanged.

Davies uses an alternating parity-based weight assignment; if the number of central nodes in the cylinder $C_n(x)$ is even, he sets $\mu_1(C_n(x))=\alpha_n$ and $\mu_2(C_n(x))= \beta_n$. If the number of central nodes is odd, however, he reverses the weights, i.e., $\mu_1(C_n(x))=\beta_n$ and $\mu_2(C_n(x))= \alpha_n$, where $\alpha_n >\beta_n >0$. To uniquely extend these to valid finite Borel measures on the product space, the mass of any cylinder at level $n$  must equal the sum of the masses of its children at level $n+1$. By construction, the space $E_{n+1}$ consists of $N_{n+1}$ central nodes and $N^2_{n+1}$ peripheral nodes. If there is an even parity at level $n$, then $C_n(x)$ takes weight $\alpha_n$ under $\mu_1$, and we find 
    \[
    \alpha_n = N_{n+1}\beta_{n+1} + N^2_{n+1} \alpha_{n+1}, 
    \]
    since transitioning to a central node flips the parity, whereas transitioning to a peripheral node leaves the parity unchanged. Similarly, if it has odd parity at level $n$, we obtain
    \[
    \beta_n = N_{n+1}\alpha_{n+1} + N^2_{n+1} \beta_{n+1}.
    \]
    This iterative construction, together with the constraint that $\alpha_n >\beta_n >0$,  requires choosing $N_n > R_{n-1}:=\alpha_{n-1}/\beta_{n-1}$, yielding $R_n= \frac{N_n R_{n-1} -1}{N_n-R_{n-1}}$.

\begin{Proposition}\label{prop:coverdavies}
    In Davies' construction the covering number grows exponentially in $1/r$ and doubly exponentially in $n$. 
\end{Proposition}

    \begin{proof}
    We only need to consider $r=1/2^n$ for $n \in \mb{N}$. Fixing $n$, we observe that the ball $\bar{B}(x,1/2^n)$ cannot contain sequences that differ from $x$ in the first $n-1$ coordinates. This means that the history consists of 
    $
    \prod_{i=1}^{n-1}|E_i|$ possible trajectories. To cover $E_n$ with the least number of balls of radius $1/2^n$ it is clear we should take the $N_n$ central balls.  Hence, $N_n \prod_{i=1}^{n-1} N_i(N_i+1)$ is the covering number for $r=1/2^n$.  Taking then the smallest value for the integer $N_n = \flo{R_{n-1}}+1$, we find that $R_n \sim(R_0)^{2^n} = (R_0)^{1/r}$. Hence, 
    \[
    N(\Omega, 1/2^n) = N_n \prod_{i=1}^{n-1} N_i(N_i+1) \sim c^{2^{n-1}} \prod_{i=1}^{n-1}c^{2^{i}} = c^{\frac{3}{2}\cdot2^n} =c^{1.5/r}.
    \]
\end{proof}

For each $n \in \mb{N}$, $(E_n,\rho_n)$ is by construction a finite metric space, and therefore it has the ball property. However, we cannot find a uniform bound on the projection covering numbers. Before showing this, we recall that the metric $\rho$ takes values in $\{0\} \cup \{1/2^m:m \in \mb{N}_0\}$. Consequently, $B(x,1/2^k)= \bar{B}(x, 1/2^{k+1})$. Let $\bar{M}(K,r,\delta)$ denote the minimal number of closed balls of radius $r$ required to cover $K$  such that the sum of their measures exceeds $\mu(K)$ by at most $\delta$, and observe that $
M_n(K_n,1/2^k, \delta)= \bar{M}_n(K_n, 1/2^{k+1},  \delta)$. 
It is therefore without loss of generality to focus on showing that there exists no uniform bound on the projection covering numbers where the cover is constructed by closed balls. 

Every base cylinder set $C_k(x)$ is compact in $\Omega$, and by  \eqref{eq:balliscylplusfringe}, we have that 
\[
\mu_1(\bar{B}(x, 1/2^k)) = \mu_1(C_k(x))+\mu_1(F_k(x))
\]
since the sets are disjoint. The fringe set at level $k$ is a non-empty union of level $k$ cylinder sets, and because the construction enforces $\alpha_k > \beta_k > 0$, every cylinder set carries strictly positive measure, ensuring that $\mu_1(F_k(x)) =: \Delta_k \ge \beta_k > 0$.  Suppose we try to cover the compact set $C_k(x)$ with balls of radius $r \ge 1/2^{k}$. Then for $\delta <\Delta_k$, 
\[
\mu_1(\bar{B}(x,r)) \ge \mu_1(C_k(x))+\Delta_k > \mu_1(C_k(x))+\delta.
\]
This shows that a tight cover with threshold $\delta$ requires covering the cylinder set with smaller balls of radius $r \le 1/2^{k+1}$. However, balls of radius $r \le 1/2^{k+1}$ can only cover sets $C_l(x)$ with $l \ge k+1$. To cover $C_{k}(x)$ with balls of such radii, we need at least one ball centered at each central node in the graph $G_{k+1}$. By construction, there are $N_{k+1}$ such nodes, implying that 
\[
\bar{M}(C_k(x),1/2^{k+1},\delta) \ge N_{k+1}.
\]
Furthermore as we take the tolerance level $\delta \downarrow 0$, we must descend the branching structure to avoid the measure on fringes at level $k+1$ and beyond, forcing us to consider coverings with radius at most $r\le 1/2^{k+\ell}$, for some $\ell \ge 2$. To cover $C_k(x)$, we then require 
\[
\bar{M}(C_k(x),1/2^{k+\ell},\delta) \ge \prod_{i=1}^{\ell} N_{k+i}.
\]
Now, by the proof of \autoref{prop:coverdavies}, we find that $N_{k+1} \sim (R_0)^{2^{k}}$, which grows doubly exponentially in $k$. Because the metric $d_{X_n}$ exactly restricts $\rho$ to the first $n$ coordinates, covering $C_k(x)$ in $\Omega$ with balls of radius $1/2^{k+\ell}$ is structurally equivalent to covering its projection in $X_{k+\ell}$. That is, 
\[
\bar{M}(C_k(x),1/2^{k+\ell},\delta) = \bar{M}_{k+\ell}(\pi_{k+\ell}(C_k(x)), 1/2^{k+\ell}, \delta).
\]
Thus, as we require tighter covers ($\ell \to \infty$), the  projection covering numbers explode ($\bar{M}_n \to \infty$). Consequently, condition $(ii)$ of \autoref{thm:ballprod} cannot be satisfied. 

\section{Generalizations of the ball property}\label{sec:sec4}

The ball property is important because it guarantees that simple, metric-based functions respecting the intrinsic structure can separate measures. Combined with uniform tightness, it therefore provides a highly tractable framework to verify weak convergence of measures on complex metric spaces. Functional limit theory for the corresponding random fields—obtained by evaluating stochastic processes on these functions—can be derived \cite{vD26}. A natural question arises: if the ball property fails, can we construct minimal, structure-preserving generalizations that still uniquely determine probability measures?

\begin{definition}[Ball property of order $n$]
Let $(S,d_S)$ be a separable metric space and consider the collection  $\textbf{B}_{\cap_n} = \Big\{\bigcap_{i=1}^n B(x_i,r_i): x_i \in S, r_i>0, i =1, \ldots, n\Big\}$. We say that $(S,d_S)$ has the \textit{$n$th order ball property } if for any probability measures $\mu$ and $\nu$ on $(S,\mc{B}(S))$, the condition
\[
\mu(A)=\nu(A)
\]
for all $A \in \textbf{B}_{\cap_n}$ implies that $\mu=\nu$.
\end{definition}

Note that since $n$ is fixed, this is obviously not a $\pi$-system per se. Related to this, we also have a natural generalization of \autoref{prop:lambdasepc}:

\begin{Proposition}\label{prop:lambdasepck}
    On a separable metric space $(S,d_S)$,  $\lambda({\textbf{B}_{\cap_n}})$ is a separating class of sets if and only if $\textbf{B}_{\cap_n}$ is a separating class of sets. 
\end{Proposition}

To see that such classes are relevant to consider, we revisit Davies' counterexample and prove the following:
\begin{thm}
    The metric space $(\Omega, \rho)$ constructed in Davies' example has the $2$nd order ball property (with respect to closed balls), that is, all probability measures on it are determined by the class $\bar{\textbf{B}}_{\cap_2}$.
\end{thm}

\begin{proof}
It suffices to show that the collection of cylinder sets $\mc{C}$ satisfies $\mc{C} \subseteq \lambda(\bar{\textbf{B}}_{\cap_2})$. We focus on cylinders of $k$-dimensional base, since $k=1$ is trivial.  Consider a sequence $z \in \Omega$ and assume first that the $k$th coordinate of $z$ satisfies $z_k \in V_C$. Set $c=z_k$. Since $N_k >R_{k-1}>1$ implies $N_k \ge 2$, the central node $c$ possesses at least two distinct peripheral nodes. Pick  two such distinct peripheral nodes $x_k, y_k \in V_P$ sharing $c$ as their parent central node. Construct $x,y \in \Omega$ such that  $x_i = y_i = z_i, 1\le i \le k-1$ and with $k$th coordinates $x_k$ and $y_k$, respectively. Then we can write 
\[
C_k(z_1, \ldots, z_k) = \bar{B}\big(x,\frac{1}{2^k}\big) \cap \bar{B}\big(y,\frac{1}{2^k}\big).
\]
Indeed, to show this note that the intersection of two balls with the same radius can be written as
\begin{align*}
    \bar{B}\big(x,\frac{1}{2^k}\big) \cap \bar{B}\big(y,\frac{1}{2^k}\big) = D_{k-1} \cap (F_1 \cup F_2 \cup F_3\cup F_4),
\end{align*}
where the sets on the right-hand side are respectively given by
\begin{align*}
    D_{k-1}& = \{\omega \in \Omega: \omega_i = x_i=y_i, 1\le i \le k-1 \}
    \\ F_1 & = \{\omega \in \Omega: \omega_k=x_k=y_k\}
    \\F_2 &=\{\omega \in \Omega: \omega_k= x_k, \omega_k \neq y_k \text{ and } \omega_k \sim y_k\}
    \\ F_3 &=\{\omega \in \Omega: \omega_k= y_k, \omega_k \neq x_k \text{ and } \omega_k \sim x_k\}
\\ F_4 &=\{\omega \in \Omega: \omega_k \neq x_k, \omega_k \neq y_k \text{ and } \omega_k \sim x_k, \omega_k \sim y_k\},
\end{align*}
and for the particular sequences $x,y \in \Omega$ constructed above, we have $D_{k-1}  \cap F_1 =D_{k-1} \cap F_2=D_{k-1} \cap F_3= \emptyset $ since $x_k$ and $y_k$ are distinct peripheral nodes, and therefore cannot be connected. In fact we get 
\[
\bar{B}\big(x,\frac{1}{2^k}\big) \cap \bar{B}\big(y,\frac{1}{2^k}\big)= D_{k-1} \cap\{\omega: \omega_k = c\} = C_k(x_1, \ldots, x_{k-1}, c).
\]
Thus,  cylinders of the form $C_k(x)$ where $x_k \in V_C$ lie in $\lambda(\bar{\textbf{B}}_{\cap_2})$. 

Consider again sequences $x,y \in \Omega$ that agree with $z \in \Omega$ on the first $k-1$ but now suppose instead that $z_k \in V_P$, $x_k \in V_C$, and $y_k \in V_P$ with $x_k \sim y_k$ (so $x_k$ is the parent node of $y_k$). Then, we find \[
\bar{B}\big(x,\frac{1}{2^k}\big) \cap \bar{B}\big(y,\frac{1}{2^k}\big) =\{\omega \in \Omega: \omega_i=x_i=y_i,1\le i \le k-1, \omega_k \in \{x_k, y_k\}\} = C_k(x) \cup C_k(y). \tageq \label{eq:ckxy}
\]
In particular, note that $D_{k-1} \cap F_4 = \emptyset$, since there is no distinct node that can connect a central  node and its peripheral node simultaneously. As we already proved that cylinders of the form  $C_k(x) =C_k(x_1, \ldots, x_{k-1}, c)$, where $c \in V_C$, are contained in $\lambda(\bar{\textbf{B}}_{\cap_2})$, we obtain from \eqref{eq:ckxy} and the definition of a lambda-system that
\[
C_k(y) =\Big(\bar{B}\big(x,\frac{1}{2^k}\big) \cap \bar{B}\big(y,\frac{1}{2^k}\big)\Big)\setminus C_k(x) \subseteq \lambda(\bar{\textbf{B}}_{\cap_2}).
\]
Consequently, $\mc{C} \subseteq \lambda(\bar{\textbf{B}}_{\cap_2})$, meaning that the latter contains a separating class and thus $\bar{\textbf{B}}_{\cap_2}$ is separating by \autoref{prop:lambdasepck}.

\end{proof}

However, it is important to mention that this generalization is most promising for infinite tree structures. More generally, when the cylinder sets are not pre-images of singletons under coordinate projections, the analysis becomes significantly more complex. We discuss this in the next section. 

\section{The taxonomy of measure-determination via balls}\label{sec:sec5}

In this article, we have studied the driving structural properties of metric spaces for Borel  probability measures to be determined by balls. We summarize the findings here, and illustrate them with an example. 

First, we showed that a separable metric space must exhibit charged dyadic shattering for the ball property to fail. While topological dyadic shattering captures the space's local combinatorial capacity, it is not sufficient to force failure, as seen by the fact that separable Banach spaces and metric spaces admitting an isometric Hilbertian embedding possess this topological capacity yet preserve measure-determination via balls through their  geometric rigidity. Indeed, Schoenberg's theorem tells us that a metric space $(S,d_S)$ is isometric to a subset of a Hilbert space if and only if $d_S^2$ is conditionally negative-definite. In the general separable Banach case, it is effectively the fact that open convex cones can be written as monotonic unions of expanding open balls around unit vectors. The latter, however, breaks as soon as we restrict ourselves to a bounded subset $V \subset E$ of a separable Banach space $(E,\|\cdot\|)$ that retains this combinatorial capacity, as spatial constraints prevent the translation of centers to infinity. Without the ability to construct separating half-spaces, the restricted collection of balls is no longer guaranteed to tightly cover arbitrary cylinder sets without capturing excess measure. Consequently, the metric outer measure is vulnerable to decoupling from the cylinder outer measure. If this happens, we are able to keep two measures identical on all balls---which are determined by specific finite-dimensional cylinder constraints---while the combinatorial capacity provides the room to change measure in the coordinate tails.

\begin{example}
    Consider the Hilbert cube $V=[0,1]^\mb{N}$ with weighted $\ell^p$ metric 
    \[
    d_{\ell^p}(x,y) =\big( \sum_{k=1}^\infty w_k |x_k-y_k|^p\big)^{\frac{1}{p} \wedge 1}
    \]
    where $\{w_k\}$ is a sequence of positive weights such that $w_k \to 0$ as $k \to \infty$. It is straightforward to verify that $|x-y|^2$ is conditionally negative-definite (c.n.d.) on $\mb{R}$. Using results from Schoenberg \cite{schoenberg1938}, the fractional power $(|x-y|^2)^{p/2}=|x-y|^p$ is also c.n.d.\ for $p \in (0,2]$. Since this property is preserved under summation with positive weights, the function 
    \[
    f(x,y) = \sum_{k=1}^\infty w_k |x_k-y_k|^p
    \] 
    is c.n.d.\ on $V$. By Schoenberg's theorem, this implies that the metric space $(V, \sqrt{f})$ is isometric to a subset of a Hilbert space. Consequently, by the results in \cite{vD26}, $(V, \sqrt{f})$ possesses the ball property. Finally, because $\sqrt{f}$ is a monotonic transformation of $d_{\ell^p}$, both metrics generate the exact same collection of open balls. Therefore, the metric space $(V, d_{\ell^p})$ inherits the ball property for any $p \in (0,2]$. 
   
    However, for $p >2$ we lack this geometric safety net.  Furthermore, because $V$ is bounded, we lack the linear half-space structure of the full Banach space. Without these structural safety nets, the topological dyadic shattering induced by the collection of balls provides the structural flexibility needed in principle to adapt a Davies-style counterexample---constructing a pair of distinct finite Borel  measures that agree on all balls while sustaining a non-zero charge. The same structural vulnerability applies when we take the weighted supremum norm ($p=\infty$).
\end{example}

\section*{Acknowledgements}
The author gratefully acknowledges support by NSF grant DMS-2311338.

\appendix
\section{Auxiliary proofs}
\begin{proof}[Proof of \autoref{lem:atomforce}]
    Let $\mc{U}=\{B \in \textbf{B} : \nu(A \cap B) = 0\}$ be the collection of balls with restricted measure zero. We can write their union as $U = \bigcup_{B \in \mc{U}} B$. By the second countability of separable metric spaces, we can write $U=\bigcup_{i =1}^\infty B_i$ for a countable subcollection $\{B_i\}_{i \in \mb{N}} \subseteq \mc{U}$. By countable subadditivity of the nonnegative finite Borel measure $\nu \vert_A$
    \[
    \nu(A \cap U) \le \sum_{i=1}^\infty \nu(A \cap B_i) =0
    \]
    which implies that $ \nu(A \cap U)=0$. Now consider $U^\complement = S\setminus U$. Since $\nu(A \cap U) =0$ we have $\nu(A \cap U^\complement)=k$. We show that $|U^\complement \cap A| = 1$. Clearly, if $U^\complement  \cap A= \emptyset$ then it has measure $0$ instead of $k$, so there must be at least one point. Suppose therefore that $U^\complement  \cap A$ contains at least two distinct points $x,y$. Since $x\notin U$, no open ball containing $x$ belongs to $\mc{U}$. Hence $\nu\big(A \cap B(x,1/n)\big) \neq 0$ and thus, for every $n \in \mb{N}$, $\nu\big(A \cap B(x,1/n)\big)=k$. 
    By continuity from above
    \[
    \nu(A \cap \{x\})= \nu\big(A \cap \bigcap_{n=1}^\infty B(x,1/n)\big)= \lim_{n \to \infty}\nu\big(A \cap B(x,1/n)\big)=k,
    \]
    and similarly $\nu(A \cap \{y\}) =k$. Since $\{x\} \cap \{y\} = \emptyset$, countable additivity yields
    \[
    \nu(A \cap (\{x\} \cup \{y\})) = 2k,
    \]
    which is a contradiction since $\nu(A)=k$. Hence $U^\complement \cap A$ consists of exactly one point $x_0$. For any Borel set $E \subseteq A$ we can write
    \[
   \nu(E)= \nu(E \cap U\cap A) + \nu(E \cap U^\complement\cap A) = \nu(E \cap U^\complement\cap A)= \nu(E \cap \{x_0\}).
    \]
    Consequently, if $x_0 \notin E$, then $\nu(E)=0$ and if $x_0 \in E$, then $\nu(E)=k$.
\end{proof}

\end{document}